\documentclass[11pt]{article}
\usepackage[utf8]{inputenc}
\usepackage{amsmath}
\usepackage{amsfonts}
\usepackage{amssymb}
\usepackage{amsthm}
\usepackage{amsrefs}
\usepackage{mathrsfs} 
\usepackage[all]{xy}
\usepackage{enumerate}

\theoremstyle{plain}
  \newtheorem{thm}{Theorem}[section]
  \newtheorem{lem}[thm]{Lemma}
  \newtheorem{prop}[thm]{Proposition}
  \newtheorem{cor}[thm]{Corollary} 
\theoremstyle{definition}
  \newtheorem{defn}[thm]{Definition}
  \newtheorem{rmk}[thm]{Remark}
  \newtheorem{ex}[thm]{Example}

\theoremstyle{plain}

\DeclareMathOperator{\im}{im}
  
\numberwithin{equation}{section}
\allowdisplaybreaks
\def\CO{\mathcal{O}}
\def\CF{\mathcal{F}}
\def\CC{\mathcal{C}}
\def\CA{\mathcal{A}}

\def\LI{{L^{-1}}}

\def\be{\beta}
\def\ga{\gamma}
\def\th{\theta}

\def\lam{\lambda}
\def\sig{\sigma}
\def\ti{\times}
\def\om{\omega}
\def\Om{\Omega}
\def\dpp{\partial_+}
\def\dpm{\partial_-}
\def\dpa{{\partial_{+A}\,}}
\def\dma{{\partial_{-A}\,}}
\def\w{\wedge}
\def\coker{\mathrm{coker~}}
\def\id{\textrm{id}}
\def\bbe{\bar{\beta}}
\def\bxi{\bar{\xi}}
\def\bsig{\bar{\sigma}}

\def\pa{\partial}
\def\bpa{{\bar \partial}}
\def\bpap{{\bar\partial_{A}}}

\newcommand\bo{\mathbf{1}}
\begin{document}

\title{
\bf\
{Symplectic Flatness and Twisted Primitive Cohomology}
}

\author{Li-Sheng Tseng and Jiawei Zhou \\
\\
}

\date{April 26, 2022}%\today}
\maketitle

\begin{abstract} 
We introduce the notion of symplectic flatness for connections and fiber bundles over symplectic manifolds.  Given an $A_\infty$-algebra, we present a flatness condition that enables the twisting of the differential complex associated with the $A_\infty$-algebra. The symplectic flatness condition arises from twisting the $A_\infty$-algebra of differential forms constructed by Tsai, Tseng and Yau.  When the symplectic manifold is equipped with a compatible metric, the symplectic flat connections represent a special subclass of Yang-Mills connections.  We further study the cohomologies of the twisted differential complex and give a simple vanishing theorem for them.

%Constructing such bundles is motivated by, but not exactly same as the twisting of de Rham and Dolbeault complexes. Instead, we twist an $A_\infty$-algebra of special differential forms constructed by Tsai, Tseng and Yau. This algebra is equivalent to a mapping cone of the de Rham complex of the symplectic manifold. Our work generalize this statement to the twisted case, and use this method to calculate the cohomology. We also explore some examples and properties of symplectically flat connections. As we will see, a symplectically flat connection is a special type of Yang-Mills connection when the bundle is equipped with a compatible metric. Separately, the tangent bundle of a K\"ahler manifold is symplectically flat implies that the base is K\"ahler-Einstein. 

\end{abstract}

\tableofcontents

\section{Introduction}
 
For a vector bundle $E$ over a smooth manifold $M$ of dimension $d$, we can study differential forms $\Om^*(M,E)$ taking values in the fiber space of $E$.  For these forms, we can write a de Rham-type complex for $\Om^*(M,E)$: 
\begin{align}\label{tdeRham}
\xymatrix@R=30pt@C=30pt{
0\; \ar[r] & \; \Om^0(M,E) \ar[r]^{d_A} &\; \Om^1(M,E) \ar[r]^{~~d_A}& ~ \ldots ~\ar[r]^{d_A~~~}&\; \Om^d(M,E)\ar[r]^{~~~d_A} & \; 0
}
\end{align}
where locally $d_A = d+A$ and $A\in \Om^1(M, \mathrm{End}\,E)$ is the connection acting on $E$.  The above complex is only a differential complex if the curvature of the connection, $F=(d_A)^2= dA + A \w A =0$.  This is the well-known fact that the de Rham complex can be {\it twisted} by a flat bundle, i.e. a bundle $E$ that allows for a connection whose curvature vanishes.   

Now suppose $M$ is a complex manifold of dimension $d=2n$ and $E$ is a complex vector bundle over $M$.  On complex manifolds, the differential forms can be decomposed into $(p,q)$ components, that is 
$$
\Om^k(M,E) = \bigoplus_{k=p+q} \, \CA^{p,q}(M,E)\,,
$$
and the exterior derivative decomposes into the Dolbeault operators, $d= \pa + \bpa$.  The decompositions of forms and exterior derivative allow us to consider a more refined complex, the Dolbeault complex, which can be twisted as well:
\begin{align}\label{tDolbeault}
\xymatrix@R=30pt@C=30pt{
0\; \ar[r] & \; \CA^{0,0}(M,E) \ar[r]^{\bpap} &\; \CA^{0,1}(M,E) \ar[r]^{~~~~\bpap}& ~ \ldots ~\ar[r]^{\bpap~~~~~}&\; \CA^{0,n}(M,E)\ar[r]^{~~~~~~~\bpap} & \; 0
}
\end{align}
where locally $\bpap= \bpa + A^{0,1}$ and $A^{0,1}$ is the $(0,1)$ component of the connection form $A$.  That the above complex is actually a differential complex imposes the condition 
\begin{align}\label{cflat}
\left({\bpap}\right)^2=(\bpa+A^{0,1} )^2 = \bpa A^{0,1} + A^{0,1} \w A^{0,1} = F^{0,2} =0\,.
\end{align}
Notice that this is a weakening of the smooth flatness condition $F=0$ in that \eqref{cflat} requires only that the $F^{0,2}$ component of the curvature vanishes.  But a complex vector bundle with a connection such that $F^{0,2}=0$ is well-known to be equivalent to the bundle $E$ having the structure of a holomorphic vector bundle.  And so in the complex case, \textit{complex flat} bundles which can twist the Dolbeault complex are just holomorphic vector bundles.

Flat bundles and holomorphic vector bundles are basic and important objects on smooth and complex manifolds, respectively.  In this paper, we are interested in exploring special vector bundles on symplectic manifolds.  Specifically, we ask a simple question: is there a \textit{symplectic flatness} condition for a vector bundle $E$ over a symplectic manifold, $(M^{2n}, \om)$?  

We can answer this question by proceeding in a similar way as in the smooth and complex cases described above.  Let us first recall that on a symplectic manifold, $(M^{2n}, \om)$, the differential forms can be expressed in a polynomial expansion in powers of $\omega$.  This decomposition of forms is commonly called the Lefschetz decomposition and given by %and applies also to forms taking values in $E$, 
\begin{align*}
\Om^k(M) = \bigoplus_{k=2r+s} \, \om^r \w P^s(M)\,,
\end{align*}
where $P^s(M)$ for $s=0,1,\ldots, n$, denotes the space of primitive forms, i.e. forms that we can not extract an $\om$ from them.  More precisely,  a form $\beta\in P^s(M)$ if there does not exist an $\xi \in \Om^{s-2}(M)$ such that $\beta = \om \w \xi$.  And besides forms, the exterior derivative, like in the complex case, also has a decomposition into two linear differential operators $(\dpp, \dpm)$ dependent on the symplectic structure \cite{TY2} 
\begin{align}\label{dcomp}
d=\dpp + \om \w \dpm
\end{align}
and with desirable properties: $(\dpp)^2 = (\dpm)^2 =0$ and $\om \w \dpp\dpm = - \om \w \dpm \dpp$.  Together, they lead to a differential complex that is elliptic (see \cite{TY2} and references therein)
\begin{align}\label{primcomplex}
\xymatrix@R=30pt@C=30pt{
0\; \ar[r] & \; P^{0}(M) \ar[r]^\dpp &\; P^{1}(M) \ar[r]^\dpp& ~ \ldots ~\ar[r]^\dpp&\; P^{n-1}(M) \ar[r]^\dpp&\; P^{n}(M) \ar[d]^{-\dpp\dpm}\\
0\; & \; P^{0}(M) \ar[l]_{-\dpm} &\; P^{1}(M)\ar[l]_{~~-\dpm}& ~ \ldots ~\ar[l]_{~~~-\dpm}&\; P^{n-1}(M) \ar[l]_{-\dpm}&\; P^{n}(M) \ar[l]_{~~~-\dpm} \; 
}
\end{align}
%\begin{align*}\xymatrix@R=30pt@C=30pt{
%0\; \ar[r] & \; P^{0} \ar[r]^\dpp &\; P^{1} \ar[r]^\dpp& ~ \ldots ~\ar[r]^\dpp&\; P^{n} \ar[r]^{\dpp\dpm}&\; P^{n}\ar[r]^\dpm &\; P^{n-1}\ar[r]^\dpm& ~\dots ~\ar[r]^\dpm&\; P^{0} \ar[r] & 0}
%\end{align*}
Hence, we can try to extend this complex to one acting on $P^*(M,E)$, the space of primitive forms with values in $E$, and also twist the operators $(\dpp, \dpm, \dpp\dpm)$ with a connection form as in \eqref{tdeRham} and \eqref{tDolbeault}.  However, the peculiar definitions of the differentials in \eqref{primcomplex} raise immediate issues. In particular, note that both $\dpm:P^s \to P^{s-1}$ and $\dpp\dpm: P^n \to P^n$ do not increase the degree of the form by one.  Can we simply twist these operators in the primitive complex above by a connection one-form?

The twisting procedure in the symplectic case turns out to be a bit more subtle.  Nevertheless, twisting the primitive complex can still be achieved because the complex has an $A_\infty$-algebra structure \cite{TTY}.  We will show in this paper, that given any $A_\infty$-module structure of differential forms taking values in $E$, there is a natural twisting of the differentials of the complex by a connection one-form.  The new twisted differentials however do not represent a deformation of the $A_\infty$-module and hence does not lead to a new $A_\infty$-module structure.  

Our result in the symplectic case is the following. We can write down a twisted primitive complex of forms taking values in $E$, if the connection one-form on $E$ satisfies the following \textit{symplectic flatness} condition: 
\begin{defn}
For $(M^{2n},\omega)$ a symplectic manifold, let $\pi:E\to M$ be a vector bundle with a connection, $d_A$ the corresponding covariant derivative and $F$ the curvature two-form.  We call the connection \textbf{symplectically flat} if
\begin{equation}\label{defsf}
\begin{split}
F& =\Phi\,\omega \\ %\,, \qquad  and \qquad
d_A\Phi&= d\Phi + [A,\Phi]=0\,,
\end{split}
\end{equation}
where $\Phi\in \Omega^0(M, \mathrm{End}\,E)$.  If such a connection exists on $E$, we say $E$ is a symplectically flat bundle.
\end{defn}
\begin{rmk}
For a principal bundle $P$
over $(M^{2n}, \om)$, we say $P$
is symplectically flat if there exists a connection form on $P$
whose curvature satisfies \eqref{defsf}.
\end{rmk}
Though written as two equations, condition \eqref{defsf} for symplectically flat is effectively just a single equation in all $2n$ dimensions.  For when $n=1$, the first equation, $F=\Phi\, \om\,$, is trivial and gives no condition.  On the other hand, when $n\geq 2$, the second equation, $d_A \Phi=0\,$, becomes unnecessary as it is implied by the first equation of \eqref{defsf} and the Bianchi identity.

\begin{table}
\centering
{\renewcommand{\arraystretch}{1.25}
\renewcommand{\tabcolsep}{.2cm}
\begin{tabular}{c|c|c|c}\label{Table1}
&Smooth $M$ & Complex $(M^{2n}, J)$& Symplectic $(M^{2n}, \om)$\\
\hline
Forms &$\Om^k$ & $\Om^k = \oplus \,\CA^{p,q}$ & $\Om^k =\oplus\; \om^r \w P^s$\\
\hline
Differential & $d$ & $d=\pa + \bpa$ & $d = \dpp + \om \w \dpm$\\
\hline
Flatness Condition &$F=0$& $ F^{0,2}=0$ & $F=\Phi\, \om,\  d_A\Phi=0$\\
\hline
Special Local Frame& $A=0$ & $A^{0,1}=0$& $A = \Phi \lambda\,$, ~$\Phi=const.$, $d\lambda=\om$
%~~(g_{\alpha\beta}=g_c\, e^{\Phi f})$ \\

%(transition functions) & $(d g_{\alpha\beta}=0)$ & $(\bpa g_{\alpha\beta}=0)$ &  $\Phi=const.$, $d\lambda=\om$,$\,g_c=const.$  \\
\end{tabular}}
\caption{A comparison of smooth, complex and symplectic flat bundles.}
%{The transition functions $g_{\alpha\beta}$ in the smooth case can be taken to be constant, in the complex case, taken to be holomorphic, and in the symplectic case, proportional to the exponential of a constant $\Phi$ times a real function $f$.}
\end{table}

To arrive at the symplectic flatness condition, we shall begin in Section 2 by reviewing the $A_\infty$-algebra structure of the primitive cochain complex of Tsai, Tseng, and Yau \cite{TTY}.  We will present a general procedure to twist the differential complex associated with any $A_\infty$-algebra.  This general procedure gives the flatness conditions for de Rham, Dolbeault and the primitive symplectic complex and results in the twisted primitive complex given in \eqref{tprimcomplex}.   

With the symplectic flatness condition established, it is helpful to have examples of symplectic flat bundles and understand some of their properties.  We will do this in Section 3.  Indeed, we will show that symplectically flat connections are a special type of Yang-Mills connections in the presence of a compatible metric on $(M^{2n}, \om)$.  Also, as shown in Table \ref{Table1}, one can choose special local frames such that the local connection form takes the form $A=\Phi\,\lambda$, where $\Phi$ is a constant matrix and $d\lambda=\om.$  We also point out an interesting relationship: when $\om$ is an integral class and we can define a circle bundle $X$ over $M$ whose Euler class is given by $\om$, (i.e. the prequantum circle bundle of $M$), the symplectic flat bundle lifts to a flat bundle on $X$.

Finally, having twisted the primitive elliptic complex, we analyze the resulting cohomologies on $P^*(M,E)$ in Section 4.   We will calculate these twisted primitive  cohomologies on $\mathbb{R}^{2n}$ and also prove a simple vanishing theorem for the twisted cohomologies when $\Phi$ is invertible.  

In this paper, we will mainly focus on symplectically flat vector bundles.  An extensive discussion of symplectically flat principal bundles including their classification will be given in a companion work \cite{TZ}.

\

\noindent{\it Acknowledgements.~} 
We would like to acknowledge the helpful input of Matthew Gibson at the onset of this project.  We are thankful to Si Li, Lihan Wang, Damin Wu and Shing-Tung Yau for helpful discussions.  The first author would also like to acknowledge the support of the Simons Collaboration Grant No. 636284.   

This paper is written for a special issue of the Journal of Geometric Analysis in honor of Peter Li.  We are fortunate to have had the opportunity to interact closely with Peter Li through our affiliations with the UC Irvine Geometry/Topology group which Peter Li shaped and built-up over twenty plus years as a leading, senior faculty at UC Irvine. Peter Li has without fail been supportive and generous with kind advice to us.  We are grateful to him for his strong encouragement in our research work.

\section{Twisting the differential complex of an $A_\infty$-algebra}

On symplectic manifolds, Tsai-Tseng-Yau \cite{TTY} showed that the primitive cochain complex in \eqref{primcomplex} can be extended to an $A_\infty$-algebra.  (More precisely, it is a commutative $A_3$ algebra.)  We will first describe this algebra below.  For ease, we will call this  algebra of differential forms on symplectic manifolds the \textbf{TTY algebra}.  We then proceed to give a heuristic description of how to twist the differential of the primitive TTY algebra.  Going further, we show in generality how the differential of any $A_\infty$-module can be twisted. The twisted differential of the TTY algebra is just a special case of this general $A_\infty$-module twisting.

\subsection{Preliminaries: TTY algebra}

We mostly follow the notations of \cite{TTY}.

As mentioned, differential forms on a symplectic manifold $(M^{2n}, \om)$ has a Lefschetz decomposition.  Any $\eta_k\in \Om^k(M)$ can be expressed as a polynomial in $\om$:
\begin{align}\label{Lefd}
\eta_k= \beta_k + \om \w \beta_{k-2} + \ldots + \om^p \w \beta_{k-2p}+ \ldots
\end{align}
where $\{\beta_k, \beta_{k-2}, \ldots, \beta_{k-2p}, \ldots\}$ are all primitive forms in $P^*(M)$ and are determined uniquely by $\eta_k$ and $\om$.  The non-degeneracy of $\om$ also allows us to define the following three operators on differential forms \cite{TTY}:
\begin{enumerate}
\item $L^p$: When $p=k\geq 0$, $L^k = \om^k\w$.  When $p=-k <0$, $L^{-k}$ removes $k$ powers of $\omega$ from a form.  For example, acting on the Lefschetz component $\om^r \w \be_s$  for $\be_s \in P^s$, $L^{-k}(\omega^{r}\wedge\be_s)=\om^{r-k}\w\be_s$ if $k\leq r < n-s+1$. When $k > r$, then $L^{-k}(\omega^r\wedge\beta_s)=0$.
\item $*_r$: For $\alpha\in\Omega^k$,  $*_r\,\alpha=L^{n-k}\alpha$.
\item $\Pi^p$: By Lefschetz decomposition, every $\eta\in\Omega^k$  can be uniquely written as $\beta+\omega^{p+1}\wedge\gamma$, where $\beta=\beta_k + \om \w \beta_{k-2} + \ldots + \om^p \w \beta_{k-2p}$ in the decomposition of \eqref{Lefd}. Then, $\Pi^p\eta=\beta$. 
\end{enumerate}
The third operator, $\Pi^p$, is a projection operator and defines the space of $p$-filtered forms: $F^p\Om^k(M)= \Pi^p\left[\Om^k(M)\right]$  for $0\leq p \leq n$. Note that the Lefschetz decomposition of $p$-filtered forms has at most terms of order $\om^p$.  

For each $p=0, \ldots, n$,  there is a TTY algebra consisting of forms in $F^p\Om^*(M)$.   In this paper, we are mainly concerned with the $p=0$ filtered case $F^{0}\Om^*(M)=P^*(M)$ which consist of just primitive forms.  To simplify notation, we will write $\Pi=\Pi^0$, i.e. the projection onto the primitive component of the Lefschetz decomposition.
 
The TTY algebra is an $A_\infty$-algebra%, which generalizes the notion of a differential graded algebra (DGA)
.  Let us recall the definition of an $A_\infty$-algebra (for a reference, see \cite{Keller}).
\begin{defn}
An \textbf{$A_\infty$-algebra} is a $\mathbb{Z}$-graded vector space 
$\mathcal{A}=\bigoplus_{k\in \mathbb{Z}} \mathcal{A}^k$ endowed with graded linear maps
$$m_k:\mathcal{A}^{\otimes k}\to \mathcal{A}\,,\qquad k\geq1$$
of degree $2-k$ satisfying
\begin{align}\label{Airel}
\sum_{r+s+t=k}(-1)^{r+st}m_{r+t+1}(\textbf{1}^{\otimes r}\otimes m_s\otimes \textbf{1}^{\otimes t})=0\,.
\end{align}
\end{defn}
The first three relations of \eqref{Airel} are the following:
\begin{align}
m_1m_1 &= 0\, \label{Ar1}\\
m_1m_2 &= m_2(m_1\otimes\textbf{1}+\textbf{1}\otimes m_1)\, \label{Ar2}\\
m_2(\textbf{1}\otimes m_2-m_2\otimes\textbf{1}) 
&= m_1m_3+m_3(m_1\otimes\textbf{1}\otimes\textbf{1}+\textbf{1}\otimes m_1\otimes\textbf{1}+\textbf{1}\otimes\textbf{1}\otimes m_1)\,.\label{Ar3}
%\\\dots \quad & \qquad\qquad \ldots
\end{align}
By the third relation, if $m_3=0$, then $m_2$, which acts as a product, is associative.  An $A_\infty$-algebra with only $\{m_1, m_2\}$ non-zero is simply a differential graded algebra (DGA).  The TTY algebra is however generally non-associative with $m_3\neq 0$, but $m_k$ for $k\geq 4=0$ can be set to zero.  Hence, it can be more precisely called an $A_3$ algebra.
 
We now define the TTY-algebra $(\mathcal{F}, m_1, m_2, m_3)$ on primitive forms $F^0\Om^*(M)=P^*(M)$.  First, the elements of the primitive TTY-algebra are those of the differential complex \eqref{primcomplex}.  The grading follows that of the complex, and to help distinguish the two sets of primitive forms, we use the $\pm$ subscript
\begin{align}\label{Fdef}
\mathcal{F} = \left\{P^0_+, P^1_+, \ldots, P^n_+, P^n_-, \ldots, P^1_-, P^0_-\right\}
\end{align}
where for $k=0, 1, \dots, n$, $P^k_+=P^k$ have grading $k$, and  $P^k_-=P^k$ have grading $2n+1-k$.  As for the $m_k$ maps, the first map, $m_1$, is just the differential of the complex \eqref{primcomplex}.

\textbf{The $m_1$ map.}
\begin{align}\label{mdef}
m_1\beta=
\begin{cases}
~~\,\partial_+\be, & \text{ for } \be\in P^k_+\,,\quad k<n, \\
-\partial_+\partial_-\be, & \text{ for } \be\in P^n_+, \\
-\partial_-\be, & \text{ for } \be\in P^k_-. %\,,\quad k<n.
\end{cases}
\end{align}
The $m_2$ map is the product operation and is dependent on the pair of primitive spaces that it acts on.  At times, we will denote it by the product symbol and write $m_2(\be,\ga)$ as $\be\times\ga$. This $\times$ operation is graded commutative.

\textbf{The $m_2$ map.}

\begin{enumerate}
\item For $\be\in P^j_+,\ga\in P^k_+$, set
$$
\be\times\ga=\Pi(\be\wedge\ga)+\Pi *_r[-dL^{-1}(\be\wedge\ga)+(\partial_-\be)\wedge\ga+(-1)^j\be\wedge(\partial_-\ga)].
$$
Note that when $j+k\leq n$, the second term is trivial, and when $j+k>n$, the first term is trivial.
\item For $\be\in P^j_+,\ga\in P^k_-$, set
$$
\be\times\ga=(-1)^j*_r[\be\wedge(*_r\ga)].
$$
\item For $\be\in P^j_-,\ga\in P^k_+$, set
$$
\be\times\ga=*_r[(*_r\be)\wedge\ga].
$$
\item For $\be\in P^j_-,\ga\in P^k_-$, set
$$
\be\times\ga=0.
$$
\end{enumerate}
The $m_3$ maps measures the non-associativity of the product $\times$.  The non-associativity only arises when all three forms in the input come from $P^*_+$.

\textbf{The $m_3$ map.}
\begin{enumerate}
\item For $\be\in P^i_+, \ga\in P^j_+, \sigma\in P^k_+$, and $i+j+k\geq n+2$, we set
$$
m_3(\be,\ga,\sigma)=\Pi*_r[\be\wedge L^{-1}(\ga\wedge\sigma)-L^{-1}(\be\wedge\ga)\wedge\sigma].
$$
%Actually, $m_3(\be,\ga,\sigma)=0$ if $i+j+k=n+1$.
\item For all other cases, we set $m_3(\be,\ga,\sigma)=0$.
\end{enumerate}
Finally, $m_k=0$ for $k\geq 4$.

\subsection{Twisting the differential of the primitive TTY algebra}

%  that it is not natural to define an But if we keep the product unchanged, i.e. $m'_2=m_2=\w$, the twisting does not preserve the DGA structure.  In particular, the Leibniz product rule condition \eqref{Ar2} fails:
%\begin{align*}
%m'_1m_2'(\eta, \xi) \neq m'_2(m'_1\eta, \xi) \pm m'_2(\eta, m'_1 \xi) \\
%(d+A)(\eta \w \xi)
%\end{align*}

We now give a heuristic description of the twisting of the primitive elliptic complex \eqref{primcomplex}.  Let $E$ be a vector bundle and consider $\Om^*(M,E)$ and $P^*(M,E)$, the space of differential forms and primitive forms, respectively, taking values in $E$.    
We start with the relation of $(\dpp, \dpm)$ with $d$ in \eqref{dcomp}.  
%Express $\eta_k \in \Om^k(M,E)$ in Lefschetz-decomposed $\eta_k = \beta_k + \om \w \beta_{k-2} +  \om^2 \w \beta_{k-4} + \ldots$, we define two operators $\Pi: \Om^k \rightarrow P^k$ and $L^{-1}: \Om^k \rightarrow \Om^{k-2}$ as follows:
%\begin{align*}
%\Pi (\eta_k) &= \beta_k \\
%L^{-1} (\eta_k) & = \beta_{k-2} +  \om \w \beta_{k-4} + \ldots
%\end{align*}
In particular, acting on primitive forms, $\dpp = \Pi\, d$ and $\dpm = L^{-1} d$.   Now, we can decompose the twisted exterior derivative into two components when acting on primitive forms $\beta \in P^s(M,E)$. Locally with $d_A = d+A\,\w\,$, we can write
\begin{align*}
d_A \beta &= (d+A\,\w)\,\beta = (\dpp + \om\w  \dpm)\,\beta + (A \w \beta)\\
&= \Pi \left[(d + A\,\w)\, \beta\right] + \om \w L^{-1}\left[(d+A \,\w)\, \beta\right] \\
&= \Pi (d_A\, \beta) + \om \w L^{-1} (d_A\, \beta)
\end{align*}
where we have noted in the second line that a primitive form wedge a one-form has a Lefschetz decomposition into two terms, and in the third line, that the decomposition is independent of the choice of local frames. %Hence, $ \Om^1(M, E) \w P^s(M,E) = P^{s+1}(M,E) \oplus \om \w P^{s-1}(M,E)$.  %We consider a vector bundle $E$ over a symplectic manifold $(M,\omega)$ whose fiber is $V$, (with a symplectically flat connection $A$).
%\begin{defn} On a bundle $E$ with symplectically flat connection $A$, the twisted differential operators $\dpa\!: P^i(M,E)\to P^{i+1}(M,E)$ and $\dma\!: P^i(M,E)\to P^{i-1}(M,E)$ are defined as follows:
%\begin{align}\label{dApmdef}
%\dpa \beta_i = \Pi\left(d+A\right)\beta_i\, , \qquad \dma \beta_i=L^{-1}\left(d+A\right)\beta_i\,,
%\end{align}
%such that
%\begin{align}\label{dAdef} 
%\left(d+ A\right)\beta_i = \dpa\beta_i + \om \w\dma\beta_i.
%\end{align}
%\end{defn}
This allows us to define the global twisted operators:
\begin{align*}
\dpa &= \Pi\, d_A : P^i(M,E)\to P^{i+1}(M,E)\\
\dma&=L^{-1} d_A: P^i(M,E)\to P^{i-1}(M,E)
\end{align*}
such that
\begin{align}\label{dAdef}
d_A = \dpa + \om \w \dma \qquad {\rm acting~on~} P^k(M,E)
\end{align}
which gives a twisted version of \eqref{dcomp}.  Locally, we have the expressions
\begin{align}\label{dApmdef}
\dpa \beta_i = \Pi\left[\left(d+A\,\w\right)\beta_i\right]\, , \qquad \dma \beta_i=L^{-1}\left[\left(d+A\,\w\right) \beta_i\right]\,.
\end{align}

Now we can express the action of $(d_A)^2$ on a primitive form in two ways.  First, note that the commutator  
\begin{align}\label{dAcom}
[d_A , \om]=[d+A,\om] =0\,. 
\end{align}
Therefore, we have
\begin{align}
(d_A)^2 \beta &= d_A \left(\dpa \beta + \om \w \dma \beta \right)\nonumber\\
&= \dpa\dpa \beta + \om \w\dma \dpa \beta + \om \w d_A \dma \beta \nonumber\\
& = \dpa\dpa \beta + \om \w \left(\dma\dpa + \dpa \dma\right)\beta + \om^2 \w \dma\dma \beta \label{dAs1}
\end{align}
Alternatively, we can also write
\begin{align}\label{dAs2}
 (d_A)^2 \beta & = F\w \beta\nonumber\\
 &=(F_0 + \om\, \Phi) \w \beta\nonumber\\
 & = F_0 \w \beta + \om \w \Phi\beta 
\end{align}
where in the second line, we have Lefschetz decomposed the curvature $F=F_0 + \om \, \Phi$ with $F_0=\Pi\,F\in P^2(M, \mathrm{End}\, E)$ and $\Phi\in \Om^0(M, \mathrm{End}\, E)$.  If $F_0=0$, then comparing \eqref{dAs1} with \eqref{dAs2} and matching the Lefschetz components, we find
\begin{align}
\left(\dpa\!\right)^2 = \left(\dma\!\right)^2 &=0\,,\label{dpmrela}\\
\om \left(\dpa\dma +\dma\dpa\right) &=\om\, \Phi\,. \label{dpmrelb}
\end{align}
This suggests the following twisted primitive complex
\begin{align}\label{tprimcomplex}
\xymatrix@R=30pt@C=30pt{
0\; \ar[r] & \; P^{0}(M,E) \ar[r]^\dpa &\; P^{1}(M,E) \ar[r]^{~~~\dpa}& ~ \ldots ~\ar[r]^{\dpa~~~~}&\; P^{n-1}(M,E) \ar[r]^\dpa&\; P^{n}(M,E) \ar[d]^{-\dpa\dma+\Phi}\\
0\; & \; P^{0}(M,E) \ar[l]_{-\dma~~~} &\; P^{1}(M,E)\ar[l]_{~~-\dma}& ~ \ldots ~\ar[l]_{~~~~-\dma}&\; P^{n-1}(M,E) \ar[l]_{-\dma~~~~}&\; P^{n}(M,E) \ar[l]_{~~~-\dma} \; 
}
\end{align}
This is a differential complex if $F$ satisfies the symplectically flat condition of $F=\Phi \om$ (i.e. $F_0 =0 $)  and $d_A\Phi =0$ given in Definition \ref{defsf}.   In particular, we write out the composition of the differential operators in the middle of the complex:
\begin{align*}
\left(-\dpa\dma+\Phi\right)\dpa \beta_{n-1} &= \left[\dpa (\dpa\dma - \Phi)+\Phi\dpa\right]\beta_{n-1} =\left[-\Phi \dpa + \dpa \Phi\right] \beta_{n-1} \\
&=\Pi\left[\left(-\Phi(d+A)+ (d+A)\Phi\right)\beta_{n-1}\right]\\
&=\Pi\left[\left(d\Phi + [A,\Phi]\right)\beta_{n-1} \right]=\Pi\left[(d_A\Phi)\w \beta_{n-1}\right]
\end{align*}
\begin{align*}
\dma\left(-\dpa\dma + \Phi\right)\beta_n &= \left(\dpa\dma -\Phi\right)\dma\beta_n + \dma\left(\Phi\beta_n\right)=\left(-\Phi \dma + \dma \Phi \right)\beta_n \\
&= -\Phi L^{-1}(d+A)\beta_n + L^{-1}(d+A)(\Phi \beta_n) \\
&=  L^{-1}\left(d\Phi + [A,\Phi]\right)\beta_n = L^{-1}\left((d_A\Phi)\w\beta_n\right)
\end{align*}
which both vanish since $d_A\Phi=d\Phi+[A,\Phi]=0\,$.

In the next subsection, we will give a more systematic description of how to obtain the twisted differentials.   We will show how all $A_\infty$-algebras can be twisted and that the symplectic flat condition needed above matches exactly the required condition for general twisting.

\subsection{Twisting the differential of an $A_\infty$-module}

We are interested to twist the primitive elliptic complex \eqref{primcomplex} in a similar manner to how the de Rham complex is twisted in \eqref{tdeRham}.  Prior to twisting, the untwisted de Rham complex together with the wedge product gives the de Rham DGA: $(\Om^*(M), m_1=d, m_2=\w)$.  In the presence of a vector bundle $E$ with fiber $V$,  twisting the complex consists of locally tensoring by $V$, i.e. $\Om^*(U)\otimes V$ and modifying the differential $m_1=d$ to $m'_1= d+A: \Om^k(U)\otimes V \to \Om^{k+1}\otimes V$.  In order that the twisted complex remains a differential complex, we obtain the condition
\begin{align*}
m'_1\circ m'_1 = (d+ A) (d+ A) = dA + A \w A = F= 0\,.
\end{align*}
Let us make two observations. First, modifying $m_1 \to m'_1$ does not result in a new DGA consisting of $(\Om^*(M)\otimes V, m'_1=d+A, m'_2)$.  Indeed, we have not and do not need to define a new product $m'_2$ on $\Om^*(M)\otimes V$.  So twisting the de Rham complex does not represent a deformation preserving the DGA structure.   Second, without modifying the maps $(m_1, m_2)$, we can tensor the de Rham DGA by matrices: $(\Om^*(U)\otimes \text{End}\ V, m_1=d, m_2=\w)$.  This is still a DGA with 
\begin{align*}
m_2(\eta_1\otimes e_1 , \eta\otimes e_2) = m_2(\eta_1, \eta_2)\otimes (e_1\cdot e_2)
\end{align*}
Note that in the context of this tensored de Rham DGA, the flatness condition can be written simply in terms of the deformed $m'_1$ map:
\begin{align*}
F=dA + A \w A = m_1(A) + m_2(A, A) = m'_1(A) =0.
\end{align*}

Now to twist the primitive elliptic complex, we first observe that any $A_\infty$-algebra also has a natural tensor product with matrices. Let $\Omega$ be an $A_\infty$-algebra. We twist $\mathcal{A}=\Omega\otimes \text{End} \ V$ where $V$, a vector space, is the fiber of $E$.
$m_k$ acting on $\Omega^{\otimes k}$ can be extended to $m_k$ acting on $\mathcal{A}^{\otimes k}$ where 
\begin{align*}
m_k(a_1 \otimes e_1,\ldots,a_k\otimes e_k) = m_k(a_1,\ldots,a_k)\otimes(e_1\circ\cdots \circ e_k).
\end{align*}

Elements of the complex are vector-valued.  They are elements of $\mathcal{B}=\Omega\otimes V$ which to be precise is an $A_\infty$-module over $\mathcal{A}$.  (See for example \cite{Keller} for the definition of an $A_\infty $ module.)  The $m_k$ acting on $\mathcal{B}$ is given by
\begin{align*}
 m_k(a_1 \otimes e_1,\ldots,a_{k-1}\otimes e_{k-1},a_k\otimes v) = m_k(a_1,\ldots,a_k)\otimes(e_1\circ\cdots \circ e_{k-1})v. 
\end{align*}

%Let $\mathcal{B}$ be an $A_\infty$-module over an $A_\infty$-algebra $\mathcal{A}$, and $A\in\mathcal{A}$ be a grading one element. It is useful to consider a given connection acting on a differential graded module of differential forms taking values on some vector bundle $E$ in this way. When $\mathcal{B}$ is the space of local sections, the connection can be represented as an operator $m'_1$ such that $m'_1(B)=m_1(B)+m_2(A,B)$ for any element $B\in\mathcal{B}$. This $A$ can be thought of as the connection 1-form. Moreover, $m'_1$ squares to zero, i.e. $m'_1\circ m'_1 =0$, if and only if $m'_1(A)=0$.

%\begin{ex}[Twisted de Rham complex]
%Suppose $E$ is a vector bundle over a manifold $M$ and admits a local trivialization over $U\subset M$. A connection acting on $\Omega^*(U,E)$ can be represented as $d+A$, where the differential operator twisted by some $A\in\Omega^*(U,\text{End }E)$. $m'_1(A)=dA+A\wedge A=0$ over any $U$ if and only if the curvature $F=0$, i.e. this connection is flat.
%\end{ex}

The above twisting for differential graded algebras can be generalized to twist any $A_\infty$-module $\mathcal{B}$.  In \cite{Gibson}, Gibson described the case for $A_3$-algebra (i.e. $m_k=0$ for all $k\geq4$) and showed that $m'_1(-)=m_1+m_2(A,-)-m_3(A,A,-)$ squares to zero when $m'_1(A)=0$.  Here, we give the general statement for any $A_\infty$-module $\mathcal{B}$. 

\begin{defn}\label{def-of-twisting-differential}
Let $\{\mathcal{A},m_k\}$ be an $A_\infty$-algebra, $\mathcal{B}$ an $A_\infty$-module over $\mathcal{A}$, and $A$ an element of $\mathcal{A}$ of grading one.  We define the operator $m'_1$ as
\begin{align*}
m'_1(B) & = m_1(B)+m_2(A,B)-m_3(A,A,B)-m_4(A,A,A,B)+\ldots \\
& = \sum \delta_k m_k(A^{\otimes (k-1)}\otimes B)
\end{align*}
where
\begin{align*}
\delta_k=(-1)^{\frac{(k-1)(k-2)}{2}}=
\begin{cases}
1, & k=4m+1 \text{ or } 4m+2,\\
-1, & k=4m+3 \text{ or } 4m.
\end{cases}
\end{align*}
\end{defn}

%\begin{rmk}
%To relate to twisting, the $\mathcal{A}$ above corresponds to $\mathcal{A}\otimes \text{End}\ V$ and $\mathcal{B}$ is the module over $\mathcal{A}\otimes \text{End}\ V$.
%\end{rmk}

%\begin{rmk}
Here, we use the notation $m'_1$ since it is obtained by twisting $m_1$, but $m'_1$ is generally not the differential of an $A_\infty$-module.  As mentioned above, we do not need morphisms $m'_k$ for $k\geq 2$ as we are not aiming to obtain a deformed $A_\infty$-structure.  The theorem below will give a sufficient condition that ensures that $m'_1\circ m'_1=0$. 
%\end{rmk}

\begin{thm}\label{twisting-differential}
$m'_1\circ m'_1=0$ if $m'_1(A)=0$.
\end{thm}

We first note a simple relation for the $\delta_k$'s.

\begin{lem}
For any $r\geq 0,s\geq 1$, we have 
$$\delta_{r+1}\delta_s=(-1)^{r(s-1)}\delta_{r+s}.$$
\end{lem}

\begin{proof}
Observe that $\delta_{r+1}\delta_s=(-1)^{\frac{r(r-1)+(s-1)(s-2)}{2}}$ and the power
$$
\frac{r(r-1)+(s-1)(s-2)}{2}=\frac{(r+s-1)(r+s-2)}{2}-r(s-1).
$$
\end{proof}

\begin{proof}[Proof of Theorem \ref{twisting-differential}.]
For any $B\in\mathcal{B}$,
$$
{m'_1}^2B=\sum_{i,j\geq1}\delta_i\delta_j m_i(A^{\otimes(i-1)}\otimes m_j(A^{\otimes(j-1)}\otimes B)).
$$
Let $r=i-1$ and $s=j$. Then $r\geq 0$, $s\geq 1$ and the sum can be written as
$$
\sum_{k=1}^{\infty}\sum_{r+s=k}\delta_{(r+1)}\delta_s m_{r+1}[A^{\otimes r}\otimes m_s(A^{\otimes(s-1)}\otimes B)].
$$
Rewrite every term as
\begin{align*}
\quad \delta_{(r+1)}\delta_s m_{r+1}[A^{\otimes r}\otimes m_s(A^{\otimes(s-1)}\otimes B)] 
& =(-1)^{rs}\delta_{r+1}\delta_s m_{r+1}(1^{\otimes r}\otimes m_s)(A^{\otimes(r+s-1)}\otimes B) \\
& =(-1)^r \delta_km_{r+1}(1^{\otimes r}\otimes m_s)(A^{\otimes(k-1)}\otimes B)
\end{align*}

By the definition of an $A_\infty$-algebra, we have
$$
\sum_{\substack{r+s=k \\ r\geq 0, s\geq 1}}(-1)^r m_{r+1}(1^{\otimes r}\otimes m_s)+\sum_{\substack{r+s+t=k \\ r\geq 0, s,t\geq 1}}(-1)^{r+st} m_{r+t+1}(1^{\otimes r}\otimes m_s\otimes 1^{\otimes t})=0.
$$
So ${m'_1}^2B$ can be described as
\begin{align*}
{m'_1}^2B &= -\sum_{k=1}^{\infty}\sum_{\substack{r+s+t=k \\ r\geq 0, s,t\geq 1}}(-1)^{r+st}\delta_k m_{r+t+1}(1^{\otimes r}\otimes m_s\otimes 1^{\otimes t})(A^{\otimes(k-1)}\otimes B) \\
&= \sum_{r\geq 0, s,t\geq 1}(-1)^{r+s(r+t)+1}\delta_{r+s+t} m_{r+t+1}(A^{\otimes r}\otimes m_s(A^{\otimes s})\otimes A^{\otimes (t-1)}\otimes B)
\end{align*}

Since $\delta_{r+t+1}\delta_s=(-1)^{(r+t)(s+1)}\delta_{r+s+t}$, we have $(-1)^{r+s(r+t)+1}\delta_{r+s+t}=(-1)^{t+1}\delta_{r+t+1}\delta_s$. When we take the sum over $s$ first, we get
$$
{m'_1}^2B=\sum_{r\geq 0, t\geq 1}(-1)^{t+1}\delta_{r+t+1}m_{r+t+1}\left(A^{\otimes r}\otimes \sum_{s\geq 1}\delta_s m_s(A^{\otimes s})\otimes A^{\otimes (t-1)}\otimes B\right).
$$
By assumption $\sum_{s\geq 1}\delta_s m_s(A^{\otimes s})=m'_1A=0$. Thus, ${m'_1}^2B=0$.
\end{proof}
The above prescription for $A_\infty$ twisting motivated by twisting the de Rham DGA gives the standard twisting for the Dolbeault complex and the primitive TTY-algebra.
\begin{ex}[Twisted Dolbeault complex]
Let $E$ be a vector bundle over a complex manifold $M$. Given a local trivialization over $U\subset M$, the $(0,1)$ part of a connection acting on $\Omega^{0,*}(U,E)$ can be represented as $\bar{\partial}+A^{0,1}$ with $A\in\Omega^1(U,\text{End }E)$. $m'_1(A^{0,1})=\partial A^{0,1}+A^{0,1}\wedge A^{0,1}=0$ over any $U$ if and only if $F^{0,2}=0$, i.e. this connection is holomorphically flat.
\end{ex}

\begin{ex}[Twisted primitive TTY-algebra]\label{symplectically flat}
Let $E$ be a vector bundle over a symplectic manifold $(M^{2n},\omega)$.  Given a local trivialization over $U\subset M$, a connection can be represented locally as $d+A$ where $A\in\Omega^1(U,\text{End }E)=P^1(U,\text{End }E)$.  The primitive complex   $P^*(U,E)$ is an $A_\infty$-module over the TTY-algebra $ P^*(U,\text{End }E)$.  Hence, we can define the operator $m'_1$ as in Definition \ref{def-of-twisting-differential} and is explicitly given by
\begin{align}\label{mpdef}
m'_1\beta=
\begin{cases}
\dpa \beta & \text{ for } \beta\in P^k_+(U,E)\,,~~ k<n\,, \\
(-\dpa\dma+\Phi)\beta & \text{ for } \beta\in P^n_+(U,E)\,, \\
-\dma \beta & \text{ for } \beta\in P^k_-(U,E)\,.
\end{cases}
\end{align}
In details for the $\beta\in P^n_+(U,E)$ case, we have
\begin{align*}
m'_1\beta &=m_1(\beta)+m_2(A,\beta)-m_3(A,A,\beta)\\
&= -\partial_+\partial_-\beta_n+\Pi \left[-dL^{-1}(A\w \beta_n)+\dpm A \w \beta_n-A\w\dpm\beta_n\right] \\
&\qquad  -\Pi\left[A\w L^{-1}(A\w \beta_n)-L^{-1}(A\w A)\beta_n\right] \\
&= -\dpa\dma\beta_n+\Pi\, L^{-1}(dA+A\w A)\beta_n =-\dpa\dma\beta_n+\Phi\,\beta_n\,.
\end{align*}
Furthermore, 
\begin{align*}
m'_1(A) =
\begin{cases}
\Pi (dA + A \w A) & \text{ for } n\geq 2\\
-d L^{-1} (dA + A \w A) + A \w L^{-1}(dA+A \w A) - L^{-1}(dA+A \w A) \w A & \text{ for } n = 1 
\end{cases}
\end{align*}
Therefore, $m'_1(A)=0$ if and only if the curvature has no primitive component, i.e.  $F=\Phi\,\omega$, and also 
$d\Phi+[A,\Phi] =0$, having noted that $L^{-1} F = \Phi$. The second equation means that $\Phi$ is covariantly constant which implies the global condition $d_A\Phi=0\,$. 
\end{ex}

\begin{rmk}
For the higher $p$-filtered TTY algebra (i.e. $p>0$), if we define $m'_1$ as in Definition \ref{def-of-twisting-differential}, then $m'_1(A)=0$ is just the the usual flat connection condition.
\end{rmk}

\section{Examples and properties of symplectically flat bundles}

We first give some simple examples of symplectically flat bundles.

\begin{ex}
When the principal bundle $P$ is rank 1, the condition of symplectically-flat becomes $F=dA=c\, \om$ for some constant $c$. Specifically, a circle bundle whose Euler class is $c\,\om$ would be symplectically flat.   
\end{ex}

\begin{ex}
A projectively flat bundle has curvature 
$$
F=c\,\om\,\mathbf{I}
$$ 
where $c$ is a constant and  $\mathbf{I}$ is the identity map over the fiber.  Hence, projectively flat bundles are  symplectically flat.
\end{ex}

\begin{ex}
When $\dim M=2$, the symplectically flat condition is exactly identical to satisfying the Yang-Mills equations in the presence of a compatible metric.  More generally, for $\dim M \geq 2$, a symplectically flat connection is always a critical point of the Yang-Mills functional with respect to a compatible metric
$$
\int_M \text{tr } (F\wedge *F)\,.
$$
Using the relation $*\,\om = \frac{1}{(n-1)!}\,\om^{n-1}$, it is straightforward to see that a symplectically flat curvature satisfies the Yang-Mills equation
\begin{align*}
d_A^* F = - * d_A * (\Phi\, \om) = - \frac{1}{(n-1)!}\,* \left((d_A \Phi) \, \om^{n-1}\right)=0\,.  
\end{align*}
Hence, symplectically flat connections are a special subset of Yang-Mills solutions.
\end{ex}

\begin{ex}
When the symplectic manifold has dimension $\dim M=4$, a symplectically flat connection satisfies the self-dual condition with respect to a compatible metric
$$
F=*F\,.
$$
%It is evident here 
Clearly here, the symplectically flat condition $F=\Phi\, \om$ is a stronger condition than the self-dual condition.  
\end{ex}

Below we give some properties of symplectically flat bundles.

\begin{prop}\label{local constant}
Suppose there is a symplectically-flat connection on a manifold $M$ with curvature $F=\Phi\,\om$. Locally, there exists a trivialization such that $\Phi$ is represented as a constant matrix, and the covariant derivative can be written as $d+\Phi\lambda$ for some local 1-form $\lambda$ satisfying $d\lambda=\omega$.
%gauge transformation $g$ such that $g\Phi g^{-1}$ is a constant. 
\end{prop}

The proof of the theorem is based on the following lemma (see for example, \cite{Ramanan} Proposition 5.8):

\begin{lem}\label{local flat}
Locally, if $d+\tilde{A}$ is a flat covariant derivative, then $\tilde{A}=g^{-1}dg$ for some matrix valued function $g$.
\end{lem}

\noindent \textit{Proof of Proposition \ref{local constant}. }
First choose local sections $\{s_1,\ldots,s_r\}$ forming a frame of $\Gamma(U,E)$, where $r$ is the rank of vector bundle $E$. Then $d_A$ can be written as $d+A$, and $\Phi$ can be represented by a matrix $\Phi_s$ with respect to this frame. Take $\lambda\in\Omega^1(U)$ such that $d\lambda=\om$. Since $d+A$ is a symplectically-flat covariant derivative, a straightforward calculation shows that
$$
d(A-\Phi_s\lambda)+(A-\Phi_s\lambda)\wedge(A-\Phi_s\lambda)=0.
$$
i.e. $d+A-\Phi_s\lambda$ is a local flat covariant derivative.

By Lemma \ref{local flat} there exists some invertible $g$ such that
$$
A-\Phi_s\lambda=g^{-1}dg.
$$
Thus,
$$
gAg^{-1}+gdg^{-1}=g\Phi_s g^{-1}\lambda.
$$
Then we have
$$
[gAg^{-1}+gdg^{-1},g\Phi_s g^{-1}]=g\Phi_s^2 g^{-1}\lambda-g\Phi_s^2 g^{-1}\lambda=0.
$$
On the other hand,
$$
d(g\Phi_s g^{-1})+[gAg^{-1}+gdg^{-1},g\Phi_s g^{-1}]=g(d\Phi_s+[A,\Phi_s])g^{-1}=0.
$$
Therefore, $d(g\Phi_s g^{-1})=0$, i.e. $g\Phi_s g^{-1}$ is a constant. Let
$$
[s'_1\,\ldots\, s'_r]=[s_1\, \ldots\, s_r]g^{-1}
$$
be another local frame. Then $\Phi$ will be represented by $g\Phi_s g^{-1}$ with respect to this new frame. And the local covariant derivative becomes
$$
d+gAg^{-1}+gdg^{-1}=d+g\Phi_s g^{-1}\lambda\,.
$$
\qed

As an application of Proposition \ref{local constant}, suppose $(M,\omega)$ is a K\"ahler manifold and its Levi-Civita connection is symplectically flat. Then the Ricci curvature Ric$(u,v)$ satisfies
$$
\text{Ric}(u,v)=i\,\text{tr}_{\mathbb{C}}F(u,v)=\frac{1}{2}\omega(u,v)(\text{tr}\,J\Phi+i\,\text{tr}\,\Phi).
$$

By Theorem \ref{local constant} around every point in $M$, we can find some $g$ such that $g\Phi g^{-1}$ is a constant. That implies that $\text{tr}\,\Phi=\text{tr}\,g\Phi g^{-1}$ is a constant. For the same reason, $\text{tr}\,J\Phi$ is also a constant. Therefore, the Ricci tensor $r(u,v)$ has the following property:
$$
r(u,v)=\text{Ric}(-Ju,v)=c\,\omega(-Ju,v)=c\,g(u,v),
$$
where $c=\frac{1}{2}(\text{tr}\,J\Phi+i\,\text{tr}\,\Phi)$. In other words, we have
\begin{prop}
If the Levi-Civita connection of a K\"ahler manifold is symplectically flat, then the manifold is K\"ahler-Einstein.
\end{prop}

Symplectically flat bundles also have a simple alternative description.  When $\om$ is integral, it induces a flat connection on the prequantum circle bundle, i.e. the circle bundle over $M$ with Euler class given by $\om$.  

To show this, we recall a result of Tanaka-Tseng \cite{TT} that the primitive TTY algebra %of $F^0\Omega^*(M)=P^*(M)$ 
is $A_\infty$-quasi-isomorphic to the cone algebra $\mathcal{C}^*(M)=\Omega^*(M)\oplus\theta\,\Omega^*(M)$ where $d\theta=\omega$.  Furthermore, when $\omega$ is integral and we can consider a circle bundle $X$ over $M$ whose Euler class is $\omega$, then the de Rham DGA of the circle bundle $\Omega^*(X)$ is quasi-isomorphic to both the $\mathcal{C}^*(M)$ algebra and the primitive TTY algebra. %$F^{0}\Omega^*(M)=P^*(M)$

%result of the following theorem from \cite{TT}:
%\begin{thm}[TT \cite{TT}]\label{TTY to circle bundle}
%Let $(M,\omega)$ be a symplectic manifold. Its filtered complex $F^p\Omega^*(M)$ is $A_\infty$-quasi-isomorphic to the cone algebra $\mathcal{C}^*_p(M)=\Omega^*(M)\oplus\theta\Omega^*(M)$ where $d\theta=\omega^{p+1}$.

%Moreover, when $\omega$ is integral. There exists an $S^{2p+1}$-bundle $X$ over $M$ whose Euler class is $[\omega^{p+1}]$. $\Omega^*(X)$ is also quasi-isomorphic to these $A_\infty$-algebras.
%\end{thm}

We can extend these quasi-isomorphism relations between algebras to include symplectically flat connections. Let $E$ be a vector bundle over $M$ with a connection, and $d_A$ the corresponding covariant derivative. On the twisted cone algebra $\mathcal{C}^*(M,E)=\Omega^*(M,E)\oplus\theta\,\Omega^*(M,E)$ with $d\theta=\om$, we define the operator
$$
D_{\mathcal{C}}=d_A-\theta\,\Phi.
$$

\begin{prop}\label{coneflat}
The above connection is symplectically flat if and only if $D_{\mathcal{C}}^2=0$.

%When $\zeta$ is integral, this is equivalent to $D_X$ is a flat connection on $X\times_M E$.
\end{prop}
\begin{proof}
Write $d_A=d+A$ locally, then $D_\mathcal{C}=d+A-\theta\Phi$.
\begin{align*}
& D_{\mathcal{C}}^2=d(A-\theta\Phi)+(A-\theta\Phi)\w(A-\theta\Phi)=0 \\
\iff \quad & (dA+A\wedge A-\Phi\om)+\theta(d\Phi+[A,\Phi])=0 \\
\iff \quad &
\begin{cases}
F=dA+A\wedge A=\Phi\om \\
d_A\Phi=d\Phi+[A,\Phi]=0
\end{cases}
\end{align*}
\end{proof}

\begin{cor}
Suppose $\om\in \Omega^2(M)$ is an integral closed 2-form on a manifold $M$, and $\pi:X\to M$ be the circle bundle whose Euler class is $\om$. If $E$ is a symplectically flat bundle over $M$, then $\pi^*E$ is a flat bundle over $X$.
\end{cor}

%\begin{cor}
%Suppose $\om\in\Omega^2(M)$ is an integral closed 2-form on a manifold $M$.  Let $\pi:X\to M$ be the circle bundle whose Euler class is $\om$, and $P$ be another circle bundle over $M$. Then $P$ is symplectically flat if and only if $\pi^*P$ is flat.
%\end{cor}

%\noindent \textit{Proof.} We can find a global angular form $\theta$ over $X$ such that $d\theta=\om$ (c.f. \cite{Bott-Tu}). When $d_A$ is the covariant derivative of symplectically-flat connection over $E$ with curvature $\Phi\om$, we have proved that $D_{\mathcal{C}}=d_A-\theta\Phi$ is a flat connection over $\pi^*E$.

%Conversely, when $D_{\mathcal{C}}$ is the covariant derivative of a flat connection over $\pi^*E$, it can be represented by a closed $1$-form over $\pi^*E$ because it is rank $1$. Let $dz$ be the $1$-form on $X$ whose direction is along the fiber $S^1$, and $\int_{S^1}dz=1$ on all the fibers. Then let $\bar{D}_{\mathcal{C}}=\int_{S^1}D_{\mathcal{C}}\wedge dz$. $\bar{D}_{\mathcal{C}}$ is invariant along fibers, gauge equivalent to $A_X$ and still closed. So we can write $\bar{D}_{\mathcal{C}}$ as $\alpha+\Phi\,\theta$, where $\alpha\in\Omega^1(M)$ and $\Phi\in\Omega^0(M)$.

%Since
%$$
%0=dA_X+A_X\wedge A_X
%=dA_X=d\alpha+\Phi\,\om-d\Phi\wedge\theta,
%$$
%and $d\alpha+\Phi\,\om\in \Omega^2(M)$, $d\Phi$ must be $0$. This implies $\Phi$ is a constant and $d\alpha=-\Phi\,\om$. Thus, $\alpha$ can be viewed as a symplectically-flat connection $1$-form over $M$.
%\qed

\section{Calculation of twisted cohomology}
In this section, we consider a vector bundle $E$ over a symplectic manifold $(M,\omega)$ with a symplectically flat connection whose curvature $F= \Phi\, \om$.  The twisted primitive elliptic complex \eqref{tprimcomplex} gives the following cohomologies:
%\begin{align*}
%PH(M,E) = \left\{PH^0_+, PH^1_+, \ldots, PH^n_+, PH^n_-, \ldots, PH^1_-, PH^0_-\right\}
%\end{align*}
%where 
\begin{align*}
PH_+^k(M, E) &= \dfrac{\ker \dpa}{\im \dpa}\, \qquad\qquad PH_+^n(M,E)=\dfrac{\ker (\dpa\dma -\Phi)}{\im \dpa}\\
PH_-^k(M, E) &= \dfrac{\ker \dma}{\im \dma}\, \qquad\qquad PH_-^n(M,E)=\dfrac{\ker \dma}{\im (\dpa\dma -\Phi)}
\end{align*}
where  $k=0, \ldots, n-1$.  To simplify notation below, we use $m'_1$ as defined in \eqref{mpdef} to denote the differentials $\{\dpa, -(\dpa\dma -\Phi), -\dma\}$ in the complex \eqref{tprimcomplex}.

By \eqref{dpmrelb}, we obtain a condition on the elements of the $PH^*(M,E)$ cohomology.
\begin{lem}\label{Phimap}
Let $\beta \in PH^*(M,E)$.  Then $\Phi \beta$ is trivial in $PH^*(M,E)$ cohomology class.  Specifically, for $\be_k\in P^k_+(M,E)$ and $0\leq k \leq n$, 
\begin{align}\label{phipt}
\Phi\be_k = \dpa\dma \be_k,
\end{align}
and for $\bbe_k\in P^k_-(M,E)$, 
\begin{align}\label{phimt}
\Phi\bbe_k = 
\begin{cases}
\dma\dpa \bbe_k, & 0 \leq k < n \\
(-\dpa\dma + \Phi)\be_k, &  k=n 
\end{cases}
\end{align}
\end{lem}
\begin{proof}
For $k < n$, note that \eqref{dpmrelb} simplifies to just $\ \dpa\dma +\dma\dpa=\Phi\ $ as $\om$ is an injective map when acting on forms of degree less than $n$.  Using this, \eqref{phipt} and \eqref{phimt} then follow immediately after imposing $m_1'(\be_k)=\dpa\be_k = 0$ or $m_1'(\bbe_k)=-\dma\bbe_k=0$.  When $k=n$, \eqref{phipt} is just $m_1'(\be_n)= 0$ and \eqref{phimt} is also trivial since $m_1'(\bbe_n)=-\dma \bbe_n=0$. 
\end{proof}

\begin{thm}\label{vanP}
If $\Phi$ is invertible, then $PH^*(M,E)=0$.
\end{thm}
\begin{proof}
Let $\Phi^{-1}$ be the inverse of $\Phi$.  We note that $d_A \Phi=0$ implies $d_A \Phi^{-1} = 0$ since
\begin{align*}
0=d_A \left(\Phi \Phi^{-1}\right)& = d\Phi\, \Phi^{-1} + \Phi \, d\Phi^{-1} + [A, \Phi]\,\Phi^{-1} + \Phi\, [A, \Phi^{-1}]
= \Phi\, d_A\Phi^{-1} \,.
\end{align*} 
Therefore, for arbitrary $\alpha\in\Omega^*(M,E)$, we have that
\begin{align*}
%d_A(\Phi^{-1}\alpha)=
(d+A)(\Phi^{-1}\alpha)= \Phi^{-1}(d+A)\alpha 
%= \Phi^{-1}d_A \alpha
\end{align*} 
It follows that both $\partial_{+A}$ and $\partial_{-A}$ also commute with $\Phi^{-1}$. So for any $\beta\in PH^*_{\pm}(M,E)$, if $m'_1\beta=0$, then $m'_1(\Phi^{-1}\beta)=\Phi^{-1}m'_1\beta=0$. By Lemma \ref{Phimap}, $\beta=\Phi(\Phi^{-1}\beta)$ must be $m'_1$-exact. 

%At every $x\in M$, there is a neighborhood $U$ of $x$ and local trivialization $\sigma:E|_U\to U\times V$ such that $\sigma^*\Phi:U\to \mathrm{End}(V)$ is a constant. For simplicity, we will just use $\Phi$ to denote $\sigma^*\Phi$. In this trivialization, $D$ is represented as $d+A$ for some $A\in\Omega^1(U,End(V))$. Since $D\Phi=0$, we have $d\Phi+[A,\Phi]=0$, then $A\Phi=\Phi A$. Thus, $A\Phi^{-1}=\Phi^{-1}A$ and so
%$$
%(d+A)(\Phi^{-1}\alpha)=\Phi^{-1}(d+A)\alpha.
%$$
%This implies $D(\Phi^{-1}\alpha)=\Phi^{-1}D\alpha$ on $U$. Since $x$ is arbitrary, this equation holds globally. It follows that both $\partial_{+A}$ and $\partial_{-A}$ commute with $\Phi^{-1}$. So for any $\beta\in PH^*_{\pm}$, if $m'_1\beta=0$, then $m'_1(\Phi^{-1}\beta)=\Phi^{-1}m'_1\beta=0$. By Corollary \ref{PhiT}, $\beta=\Phi(\Phi^{-1}\beta)$ must be $m'_1$-exact. 
\end{proof}
\begin{cor}
When $\mathrm{rank}~E = 1$ and $E$ is non-flat, then $PH^*(M,E)=0$.
\end{cor}

The above Theorem \ref{vanP} is a vanishing statement for $PH^*(M,E)$ of which $\Phi$ plays a central role.

\subsection{Local cohomologies}
For arbitrary $x\in M$, we have a neighborhood $U$ of $x$ isomorphic to $\mathbb{R}^{2n}$ such that $E|_U\simeq U\times V$. There exist $\lambda\in\Omega^1(U)$ such that $d\lambda=\omega$. According the proof of Proposition \ref{local constant}, we can find a frame $\{e_i\}$ on $E|_U$ such that $Ae_i=\Phi\lambda e_i$ and $\Phi\in \text{End }V$ is a constant.  Locally, we have $A=\Phi \lam$ and we obtain the following result regarding the twisted primitive cohomologies:
\begin{thm}\label{localP}
\begin{align*}
PH^k_+(U,E)&=
\begin{cases}
\ker \Phi,& k=0 \\
\lambda\ \mathrm{coker}\,\Phi,& k=1 \\
0,& k\geq 2
\end{cases}
\\  PH^k_-(U,E)&=0\,.
\end{align*}
Here, $\ker \Phi$ and  $\,\mathrm{coker}\,\Phi$ are subspaces of $V$ and can be represented by constant sections. %i.e. $PH^0_+(U,E),PH^1_+(U,E)$ can be represented by constants.
\end{thm}

We will make use for the proof of the theorem the local Poincare' lemmas for the $\{\dpp,\dpp\dpm, \dpm\}$ operators in \cite{TY2}.  Before giving the proof, let us first write down some expressions that are $m_1$-closed, where $m_1$ refers to the the untwisted differentials of \eqref{mdef}, instead of $m_1'$-closed.  Below, we will often use the simpler product notation $\times$ to denote the $m_2$ map as described in Sec. 2.1.

\begin{lem}\label{closedlem}
Let $\be_k\in P^k_+$ such that $m_1'(\be_k)=0$ and $\bbe_k\in P^k_-$ such that $m_1'(\bbe_k)=0$.  Then for a symplectically flat connection of the form $A=\Phi\lam$ where $\Phi$ is a constant,
\begin{align*}
\dpp\left(\be_k - \lam \ti \dma \be_k\right)&=0\,, \qquad k=0, 1, \ldots, n-1,\\
-\dpp\dpm\left(\be_n - \lam \ti \dma \be_n\right)&=0\,,\\
-\dpm\left(\bbe_k + \lam \ti \dpa\bbe_k\right)&=0\,,\qquad k=0, 1, \ldots, n-1.
\end{align*}
\end{lem}
\begin{proof}
For $\be_k\in P^k_+$ and $m_1'(\be_k)=0$, \eqref{phipt} implies $\Phi \be_k = (\dpp + \Phi\lam \ti) \dma \be_k$, or equivalently,
\begin{align}\label{closed1}
\Phi\left(\be_k - \lam \ti \dma \be_k\right) = \dpp \dma \be_k.
\end{align}
By direct computation, we will show $\be_k - \lam \ti \dma \be_k$ is $m_1$-closed.  When $k<n$, we have
\begin{align*}
\dpp\left(\be_k - \lam \ti \dma \be_k\right) 
&= \dpp \be_k + \lam \ti \dpp\dma\be_k\\
&=\dpp \be_k + \Phi \lam \ti (\be_k- \lam \ti \dma \be_k)\\
&= \dpa \be_k =0
\end{align*}
having used Leibniz rule in the first line and noting that $\lam \ti (\lam \ti \dma \be_k) = (\lam \ti \lam) \ti \dma \be_k = 0$. When $k=n$, we have
\begin{align*}
-\dpp\dpm\left(\be_n - \lam \ti \dma \be_n\right) 
&= -\dpp\dpm \be_n + \lam \ti \dpp \dma \be_n\\
&= - \dpp\dpm \be_n + \Phi\lam \ti  \be_n - \Phi\lam \ti \left(\lam \ti \dma \be_n\right)\\
&= m_1(\be_n) + m_2(\Phi\lam, \be_n) -m_3(\Phi \lam, \Phi\lam, \be_n)=-(\dpp\dpm)_A \be_n =0
\end{align*} 
which follows from the relation
\begin{align*}
\Phi\lam \ti (\lam \ti \dma \be_n)
&=(\Phi\lam \ti \lam) \ti \dma \be_n + m_3(\Phi\lam, \lam, \dpp\dma\be_n)\\
&= m_3(\Phi\lam, \Phi\lam, \be_n - \lam\ti \dma\be_n)= m_3(\Phi\lam, \Phi\lam, \be_n)
\end{align*}
since
\begin{align*}
m_3(\Phi\lam, \Phi\lam, m_2(\lam, \dma\be_n))= m_2 (\Phi\lam, m_3(\Phi\lam, \lam, \dma\be_n)) =0,
\end{align*}
having used the following $A_\infty$ relation involving $m_4$ (with $m_4=0$ for the TTY algebra):
\begin{align*}\begin{split}
m_2 (m_3 \otimes \bo)& + m_2 ( \bo \otimes m_3) - m_3 (m_2 \otimes \bo^{\otimes 2}) + m_3 ( \bo \otimes m_2 \otimes \bo) - m_3 ( \bo^{\otimes 2} \otimes m_2) \\
& = m_1 m_4 - m_4\left( m_1 \otimes \bo^{\otimes 3} + \bo \otimes m_1 \otimes \bo^{\otimes 2} + \bo^{\otimes 2} \otimes m_1 \otimes \bo + \bo^{\otimes 3} \otimes m_1\right)=0\,.
\end{split}\end{align*}
For $\bbe_k\in P^k_-$ with $k\neq n$ and $m_1(\bbe_k)=0$, \eqref{phimt} implies $\Phi \bbe_k = (\dpm - \Phi \lam \ti) \dpa \bbe_k$, or equivalently, 
\begin{align}\label{closed2}
\Phi \left(\bbe_k + \lam \ti \dpa\bbe_k\right) = \dpm \dpa \bbe_k.
\end{align}
Similar to above, it then follows
\begin{align*}
\dpm \left(\bbe_k + \lam \ti \dpa\bbe_k\right) 
&= \dpm \bbe_k - \lam \ti \dpm\dpa\bbe_k \\
&=\dpm \bbe_k - \Phi \lam \ti \bbe_k - \Phi\lam \ti (\lam \ti \dpa \bbe_k) = \dma \bbe_k =0
\end{align*}
having noted that $\lam \ti (\lam \ti \dpa \bbe_k) = (\lam\ti\lam)\ti \dpa\bbe_k =0$.
\end{proof}

We now give a proof of Theorem \ref{localP}.

\begin{proof}[Proof of Theorem \ref{localP}]
The proof of the theorem is divided into five cases.

\textbf{Case 1:} $PH^0_+(U,E)$. Let $\be_0\in PH^0_+(U,E)$.   By \eqref{phipt}, we have $\Phi \be_0 =0$ and hence, $\be_0 \in \ker \Phi$.  Imposing $\dpa\be_0=0$, we find 
\begin{align}
0=\dpa \be_0 = (d + \Phi \lam) \be_0 = d\be_0\,.
\end{align}  
Therefore, $\be_0$ must be constant and an element of $\ker \Phi$. 

\textbf{Case 2:} $PH^1_+(U,E)$. Let $\be_1\in PH^1_+(U,E)$.  By \eqref{phipt}, $\Phi \be_1 =  (d + \Phi \lam) \dma\be_1$ which we can write as
\begin{align}\label{loc1}
\Phi \left(\be_1-\lam\,\dma\be_1\right)=d\,\dma\be_1\,.
\end{align}
We can show that $(\be_1-\lam\,\dma\be_1)$ is $d$-closed:
\begin{align*}
d(\be_1-\lam\,\dma\be_1)
&=d\be_1 - \om\, \dma\be_1 + \lam \w d\,\dma\be_1\\
&=d\be_1 - \om\, \dma\be_1 + \Phi\lam \w (\be_1 - \lam\, \dma\be_1)\\
&= \dpa \be_1 =0
\end{align*}
By Poincar\'e Lemma, there exists $\xi_0\in P^0_+$ such that 
\begin{align}\label{loc2}
\be_1-\lam\,\dma\be_1= d\xi_0.
\end{align}  
Together with \eqref{loc1}, this implies
\begin{align}\label{loc3}
\dma\be_1 = \Phi \xi_0 + \sig_0
\end{align}
where $\sig_0$ is a constant matrix.  Substituting \eqref{loc3} into \eqref{loc2} gives
\begin{align*}
\be_1 = d \xi_0 + \lam (\Phi \xi_0 + \sig_0)= \dpa \xi_0 +  \lam\, \sig_0.
\end{align*}
But if $\sig_0\in \im\Phi$, then there exists some constant $\tilde{\sig_0}$ such that $\lam\,\sig_0 = (d+ \Phi \lam)\tilde{\sig_0}$.  Therefore, $\sig_0\in \coker \Phi$ if $\be_1=\lam\,\sig_0$ represents a non-trivial class.

\textbf{Case 3:} $PH^k_+(U,E)$ for $k=2, \ldots, n$.  Let $\be_k \in PH^k_+(U,E)$ for $k=2, \ldots, n$.  By Lemma \ref{closedlem} and local Poincar\'e lemmas for $\dpp$ and $\dpp\dpm$ operators \cite{TY2}, there exists a $\xi_{k-1}\in P^{k-1}_+$ such that 
\begin{align}\label{loc4}
\be_k - \lam\ti \dma \be_k  = \dpp \xi_{k-1}.
\end{align}
Inserting this into \eqref{closed1} then implies
\begin{align}\label{loc5}
\dma\be_k = \Phi \xi_{k-1} + \dpp \sig_{k-2}
\end{align}
for some $\sig_{k-2}\in P^{k-2}_+$.  Together, \eqref{loc4}-\eqref{loc5} give us 
\begin{align*}
\be_k &= \dpp \xi_{k-1} + \Phi \lam \ti \xi_{k-1} + \lam \ti \dpp \sig_{k-2}\\
&=\dpa \xi_{k-1} - \dpa(\lam \ti \sig_{k-2}) + \Phi \lam \ti ( \lam \ti \sig_{k-2})\\
&=\dpa(\xi_{k-1} - \lam \ti \sig_{k-2})
\end{align*}
after noting that $\lam \ti (\lam \ti \sig_{k-2}) = (\lam \ti \lam) \ti \sig_{k-2} =0.$

\textbf{Case 4:} $PH^k_-(U,E)$ for $k=0, 1, \ldots, n-1$.
Let $\bbe_k \in PH^k_+(U,E)$ for $k=0,1, \ldots, n-1$.  By Lemma \ref{closedlem} and the local Poincar\'e lemmas for $\dpm$ operator \cite{TY2}, there exists a $\bxi_{k+1}\in P^{k+1}_-$ such that 
\begin{align}\label{loc6}
\bbe_k + \lam\ti \dpa \bbe_k  = \dpm \bxi_{k+1}.
\end{align}
Inserting this into \eqref{closed2} then implies
\begin{align}\label{loc7}
\dpa\bbe_k = 
\begin{cases}
\Phi \bxi_{k+1} + \dpm \bsig_{k+2} &  k=0,1, \ldots, n-2\\
\Phi \bxi_{k+1} + \dpp\dpm \sig_n &  k = n-1
\end{cases}
\end{align}
for some $\bsig_{k+2}\in P^{k+2}_-$ and $\sig_n\in P^n_+$.  For $k<n-1$, \eqref{loc6}-\eqref{loc7} imply
\begin{align*}
\bbe_k &= \dpm \bxi_{k+1} - \Phi \lam \ti \bxi_{k+1} - \lam \ti \dpm \bsig_{k+2}\\
&=\dma \bxi_{k+1} + \dma(\lam \ti \bsig_{k+2}) + \Phi \lam \ti ( \lam \ti \bsig_{k+2})\\
&=\dma(\bxi_{k+1} + \lam \ti \bsig_{k+2})
\end{align*}
since $\lam \ti (\lam \ti \bsig_{k+2}) = (\lam \ti \lam) \ti \bsig_{k+2} =0$. Similarly, when $k=n-1$, we obtain 
\begin{align*}
\bbe_k &= \dpm \bxi_{n} - \Phi \lam \ti \bxi_{n} - \lam \ti \dpp\dpm \sig_{n}\\
&=\dma \bxi_{n} + \dma(\lam \ti \sig_{n}) + \Phi \lam \ti ( \lam \ti \sig_{n})\\
&=\dma(\bxi_{n} + \lam \ti \sig_{n})
\end{align*}
having noted that $\lam \ti (\lam \ti \sig_n) = (\lam\ti\lam)\ti \sig_n+ m_3(\lam, \lam, -\dpp\dpm \sig_n)$=0.

\textbf{Case 5:} $PH^n_-(U,E)$.  Let $\bbe_n \in PH^n_-(U,E)$.  Hence, $0=\dma \bbe_n = \dpm \bbe_n - \Phi \lam \ti \bbe_n$, or equivalently, 
\begin{align}\label{loc8}
\dpm \bbe_n= \Phi \lam \ti \bbe_n  
\end{align}
Further, by Leibniz rule, we find that $\lam \ti \bbe_n$ is $\dpm$-closed. 
\begin{align*}
\dpm(\lam \ti \bbe_n) = - \lam \ti \dpm \bbe_n = -\lam \ti (\Phi\lam \ti \bbe_n) = 0 \,.
\end{align*}
Since $\lam \ti \bbe_n\in P^{n-1}_-$ is $\dpm$-closed, local Poincar\'e lemma implies there exists an $\bxi_n\in P^n_-$ such that 
\begin{align}\label{loc9}
\lam \ti \bbe_n = \dpm \bxi_n\,.
\end{align}  
Together with \eqref{loc8}, this implies
\begin{align}\label{loc10}
\bbe_n= \Phi \bxi_n + \dpp\dpm \sig_n
\end{align}
for some $\sig_n\in P^n_+$.  Now let $\xi_n\in P^n_+$ be the same $n$-form as $\bxi_n$, i.e. $\xi_n=\bxi_n$ as primitive form.  We will show that in fact $\bbe_n = m_1'(\xi_n - \lam \ti \dpm \sig_n) = (-\dpa\dma + \Phi) (\xi_n - \lam \ti \dpm \sig_n)$ with $(\xi_n - \lam \ti \dpm \sig_n)\in P^n_+$.  To do so, we write $\lam \ti \dpm \sig_n = \Pi(\lam \w \dpm \sig_n)= \lam \w \dpm\sig_n - \om L^{-1}(\lam \w \dpm\sig_n)$.
\begin{align*}
(-\dpa&\dma + \Phi) (\xi_n - \lam \ti \dpm \sig_n)\\
 &= \left[-(d+\Phi \lam\w )L^{-1}(d+\Phi \lam\w) + \Phi\right]\left[\xi_n - (\lam \w \dpm\sig_n - \om L^{-1}(\lam \w \dpm\sig_n))\right]\\
 &= \left[-(d+\Phi \lam\w )L^{-1}(d+\Phi \lam\w) + \Phi\right](\xi_n - \lam \w \dpm\sig_n) 
\\ & \quad
+  \left[-(d+\Phi \lam\w )L^{-1}(d+\Phi \lam\w) + \Phi\right] \om L^{-1}(\lam \w \dpm\sig_n)
\end{align*}
Note first that the second term vanishes.
\begin{align*}
[-(d&+\Phi \lam\w )L^{-1}(d+\Phi \lam\w) + \Phi ] \om L^{-1}(\lam \w \dpm\sig_n)\\
&=-(d+\Phi \lam\w )L^{-1}\om(d+\Phi \lam\w) L^{-1}(\lam \w \dpm\sig_n) + \Phi\om L^{-1}(\lam \w \dpm\sig_n)\\
&= -(d+\Phi \lam\w )^2 L^{-1}(\lam \w \dpm\sig_n) + \Phi\om L^{-1}(\lam \w \dpm\sig_n) =0
\end{align*}
And the first term gives the desired result.
\begin{align*}
&\quad \left[-(d+\Phi \lam\w )L^{-1}(d+\Phi \lam\w) + \Phi\right](\xi_n - \lam \w \dpm\sig_n)\\
&=-(d+\Phi \lam\w )\left[\dpm \xi_n+ L^{-1}(\Phi \lam \w \xi_n)+ L^{-1}(-\om\w \dpm \sig_n + \lam \w \dpp\dpm \sig_n)\right] + \Phi \xi_n - \Phi\lam \w \dpm \sig_n\\
&=-(d+\Phi \lam\w )(-\dpm\sig_n) + \Phi \bxi_n - \Phi \lam \w \dpm\sig_n \\
&= \Phi\bxi_n + \dpp\dpm \sig_n = \bbe_n
\end{align*}
where to obtain the third line, we noted $\xi_n=\bxi_n$ as forms and applied the relations in \eqref{loc9} and \eqref{loc10}.
\end{proof}

\subsection{Global  cohomologies and relation to the twisted cone complex}
Tanaka-Tseng in \cite{TTY} gave a homotopy equivalence between the cochain complex of $P^*(M)$ and the cone complex of $\mathcal{C}^*(M)=\Om^*(M) \oplus \theta \, \Om^*(M)$.   For differential forms taking values in $E$, a symplectically flat vector bundle over a symplectic manifold $(M,\omega)$,  we will construct here a similar relation.

Primitive forms with values in $E$, denoted by $P^*(M,E)$, form a twisted cochain complex with differential $m_1'$ as in \eqref{tprimcomplex}. Similar to \eqref{Fdef}, we use $\mathcal{F}^*(M,E)$ to denote this complex, i.e.
$$
\mathcal{F}^j(M,E)=
\begin{cases}
P^j_+(M,E), & 0\leq j\leq n, \\
P^{2n+1-j}_-(M,E), & n+1\leq j\leq 2n+1.
\end{cases}
$$
For the cone, the differential forms with values in $E$
\begin{align}\label{conef}
\mathcal{C}^j(M,E)=\Omega^j(M,E)\oplus\theta\,\Omega^{j-1}(M,E)
\end{align}
also form a twisted cochain complex, with differential $D_{\mathcal{C}}= d_A-\theta\Phi$.  Recall from Proposition \ref{coneflat} that $D_{\mathcal{C}}^2=0$ as long as $d_A$ is symplectically flat.

It is useful to decompose each $\mathcal{C}^j(M,E)$ into primitive components.  Specifically, let $k\leq n$. For $\alpha_k\in \mathcal{C}^k(M,E)$, \eqref{conef} implies the following decomposition
\begin{align}\label{conek1}
\alpha_k %&= \sum_{r\geq 0}\omega^r\wedge\beta_{k-2r}+\theta\sum_{r\geq 0}\omega^r\wedge\beta_{k-2r-1} 
%\\&
&=\eta_k + \theta\, \eta_{k-1}\nonumber\\
&= \beta_k+\omega\wedge\beta_{k-2}+\cdots+\theta\left(\beta_{k-1}+\omega\wedge\beta_{k-3}+\ldots\right).
\end{align}
 For $\alpha_{2n+1-k}\in \mathcal{C}^{2n+1-k}(M,E)$, noting that $\mathcal{C}^{2n+1-k}(M,E)
=\Omega^{2n+1-k}(M,E)+\theta\big( *_r\Omega^k(M,E)\big)$, we have the following  %[\textit{still need to define $*_r$}]
\begin{align}\label{conek2}
\alpha_{2n+1-k} %&= *_r\sum_{r\geq 0}\omega^r\wedge\beta_{k-2r-1}+\theta(*_r\sum_{r\geq 0}\omega^r\wedge\beta_{k-2r}) 
%\\&
&=\eta_{2n+1-k} + \theta\,\eta_{2n-k}\nonumber\\
&= \om^{n-k+1}\w\left(\beta_{k-1}+\omega\wedge\beta_{k-3}+\ldots\right)+\theta\, \om^{n-k}\w\left(\beta_k+\omega\wedge\beta_{k-2}+\ldots\right).
\end{align}
Throughout this subsection, all $\beta_i\in P^i(M,E)$ are primitive forms with values in $E$.  These primitive forms can be acted upon by $d_A$.

We now define two maps.  
\begin{defn}\label{fgdef}
In terms of the decompositions \eqref{conek1}-\eqref{conek2}, we define the map $f:\mathcal{C}^j(M,E)\to\mathcal{F}^j(M,E)$ 
\begin{align*}
f(\alpha_j) = 
\begin{cases}
\quad\ \beta_j  & 0\leq j\leq n \\
-\left(\beta_k +\dpa\beta_{k-1}\right) & n+1\leq j\leq 2n+1,\, k=2n+1-j
\end{cases}
\end{align*}
and $g:\mathcal{F}^j(M,E)\to\mathcal{C}^j(M,E)$ 
\begin{align*}
g(\beta_j) = 
\begin{cases}
\beta_j-\theta \,\dma\beta_j & 0\leq j\leq n \\
%-\theta *_r\beta_k=
-\theta\,\omega^{n-k}\wedge\beta_k & n+1\leq j\leq 2n+1,\, k=2n+1-j
\end{cases}
\end{align*}
\end{defn}
%Note when $n+1\leq j\leq 2n+1$, $\mathcal{F}^j(M,E)=P^k_-(M,E)$.

In the next two lemmas, we will show that both $f$ and $g$ are chain maps.  
\begin{lem}
The map $f$ is a chain map, i.e the following graph commutes for all $0\leq j\leq 2n$.
$$
\xymatrix
{
\mathcal{C}^j \ar[r]^-{D_{\mathcal{C}}} \ar[d]_-f & \mathcal{C}^{j+1} \ar[d]^-f  \\
\mathcal{F}^j \ar[r]^-{m'_1} & \mathcal{F}^{j+1}
}
$$
\end{lem}
\begin{proof}
\noindent \textbf{Case 1.} $j<n$.
%For arbitrary $\alpha_j\in\mathcal{C}^j$, we can write it as $\alpha_j=\beta_j+\omega\wedge\xi+\theta\eta$ where $\beta_j\in PH^j(M,E)$. 
For $\alpha_j\in\mathcal{C}^j$, we have $f(\alpha_j)=\beta_j$ and  $m'_1\circ f(\alpha_j)=\dpa\beta_j$.  On the other hand, 
\begin{align*}
D_{\mathcal{C}}(\alpha_j)=(d_A)\eta_j + \om \w\eta_{j-1} -\theta\left[(d_A)\eta_{j-1} + \Phi\, \eta_j\right] 
%= \dpa \be_j + \CO(\om) - \th [ \ldots ]
%(d+A)\beta_j+\omega\wedge(d\xi+A\xi-\eta)-\theta(\Phi\beta_j+\omega\wedge\xi+d\eta+A\eta).
\end{align*}
Therefore, $f\circ D_{\mathcal{C}}(\alpha_j)=\Pi(d_A)\beta_j=m'_1\circ f(\alpha_j)$.

\noindent \textbf{Case 2.} {\bf $j=n$.}
For $\alpha_n\in\mathcal{C}^n$, 
%we can write it as $\alpha_n=\beta_n+\omega\wedge\beta_{n-2}+\omega^2\wedge\xi+\theta(\beta_{n-1}+\omega\wedge\eta)$. 
%Here $\beta_i\in PH^i(M,E)$ for $i=n-2,n-1,n$. 
we have $f(\alpha_n)=\beta_n$ and
\begin{align*}
m'_1\circ f(\alpha_j) =\left(-\dpa\dma+\Phi\right)\beta_n.
\end{align*}
On the other hand,
\begin{align*}
D_{\mathcal{C}}(\alpha_n)&=(d_A)\eta_n + \om \w\eta_{n-1} -\theta\left[(d_A)\eta_{n-1} + \Phi\, \eta_n\right]\\
&=\left[\om \w \left(\dma\beta_n + \dpa\beta_{n-2}+\beta_{n-1}\right) + \CO(\om^2)\right] -\theta \left[ \left(\dpa\beta_{n-1} + \Phi \beta_n\right) +\CO(\om)\right]\\
%&= (d+A)\beta_n+\omega\wedge[(d+A)\beta_{n-2}+\beta_{n-1}]+\omega^2\wedge[(d+A)\xi+\eta] \\
%& \qquad -\theta[\Phi\beta_n+(d+A)\beta_{n-1}+\Phi(\omega\wedge\beta_{n-2})+(d+A)(\omega\wedge\eta)+\Phi(\omega^2\wedge\xi)] \\
%&= *_r[L^{-1}(d+A)\beta_n+(d+A)\beta_{n-2}+\beta_{n-1}+\omega\wedge((d+A)\xi+\eta)] \\
%& \qquad -\theta *_r[\Phi\beta_n+(d+A)\beta_{n-1}+\omega\wedge(\Phi\beta_{n-2}+(d+A)\eta+\omega\wedge\Phi\xi)].
\end{align*}
Hence,
\begin{align*}
f\circ D_{\mathcal{C}}(\alpha_n)&= 
\dpa\be_{n-1} + \Phi\be_n - \dpa\!\left(\dma\be_n + \dpa\be_{n-2}+\be_{n-1}\right)= \left(-\dpa\dma + \Phi\right)\be_n
%\Phi\beta_n+\Pi^0(d+A)\beta_{n-1}-\Pi^0(d+A)[L^{-1}(d+A)\beta_n+\Pi^0(d+A)\beta_{n-2}+\beta_{n-1}] \\
%&= \Phi\beta_n-\Pi^0(d+A)L^{-1}(d+A)\beta_n-\Pi^0(d+A)\Pi^0(d+A)\beta_{n-2}
\end{align*}
Therefore,  $m'_1\circ f(\alpha_n)=f\circ D_{\mathcal{C}}(\alpha_n)$.

\noindent \textbf{Case 3.} $j>n$.
For $\alpha_{2n+1-k}\in\mathcal{C}^j$, with $k=2n+1-j$, we have  $f(\alpha_{2n+1-k})=-\beta_k-\dpa\beta_{k-1}$, and
\begin{align*}
m'_1\circ f(\alpha_{2n+1-k}) &= \dma \be_k + \dma \dpa\be_{k-1}= \dma\be_k -\dpa\dma\be_{k-1} +\Phi \be_{k-1}.
\end{align*}
having used \eqref{dpmrelb}.  On the other hand,
\begin{align*}
D_{\mathcal{C}}(\alpha_{2n+1-k})&= \om^{n-k+2}\left[\left(\dma \be_{k-1}+\dpa\be_{k-2}+\be_{k-2}\right)+\CO(\om)\right] \\ 
&\qquad - \th \, \om^{n-k+1} \left[\left(\dma\be_k + \dpa\be_{k-2} + \Phi\be_{k-1}\right) + \CO(\om)\right] \\
%\\&= \omega^{n-k+1}\wedge[(d+A)\beta_{k-1}+\beta_k]+\omega^{n-k+2}\wedge[(d+A)\beta_{k-3}+\beta_{k-2}]+\omega^{n-k+3}\wedge [(d+A)\xi+\eta] \\
%& -\theta \{ \omega^{n-k}\wedge(d+A)\beta_k+\omega^{n-k+1}\wedge[\Phi\beta_{k-1}+(d+A)\beta_{k-2}]+\omega^{n-k+2}\wedge[\Phi\beta_{k-3}+(d+A)\eta] \\
%& +\omega^{n-k+3}\wedge\Phi\xi \} \\
%&= *_r \{ L^{-1}(d+A)\beta_{k-1}+(d+A)\beta_{k-3}+\beta_{k-2}+\omega\wedge [(d+A)\xi+\eta] \} \\
%& -\theta *_r \{ L^{-1}(d+A)\beta_k+\Phi\beta_{k-1}+(d+A)\beta_{k-2}+\omega\wedge[\Phi\beta_{k-3}+(d+A)\eta+\omega\wedge\Phi\xi] \}
%\end{align*}
%Hence,
%\begin{align*}
f\circ D_{\mathcal{C}}(\alpha_{2n+1-k}) &= \left(\dma\be_k + \dpa\be_{k-2} + \Phi\be_{k-1}\right) -\dpa\left(\dma \be_{k-1}+\dpa\be_{k-2}+\be_{k-2}\right)\\
&=\dma\be_k +\Phi\be_{k-1} -\dpa\dma\be_{k-1}
%\\
%&= L^{-1}(d+A)\beta_k+\Phi\beta_{k-1}+\Pi^0(d+A)\beta_{k-2}-\Pi^0(d+A)[L^{-1}(d+A)\beta_{k-1}+\Pi^0(d+A)\beta_{k-3}+\beta_{k-2}] \\
%&= L^{-1}(d+A)\beta_k+\Phi\beta_{k-1}-\Pi^0(d+A)L^{-1}(d+A)\beta_{k-1}-\Pi^0(d+A)\Pi^0(d+A)\beta_{k-3}.
\end{align*}
proving $m'_1\circ f(\alpha_{2n+1-k})=f\circ D_{\mathcal{C}}(\alpha_{2n+1-k})$.
\end{proof}

\begin{lem}
The map $g$ is a chain map, i.e the following graph commutes for all $0\leq j\leq 2n$.
$$
\xymatrix
{
\mathcal{C}^j \ar[r]^-{D_{\mathcal{C}}} & \mathcal{C}^{j+1}  \\
\mathcal{F}^j \ar[r]^-{m'_1} \ar[u]^g & \mathcal{F}^{j+1} \ar[u]_g
}
$$
\end{lem}
\begin{proof}
\noindent \textbf{Case 1}. $j<n$.
For $\beta_j\in\mathcal{F}^j$, $g(\beta_j)=\beta_j-\theta \dma \beta_j$, and 
\begin{align}\label{gc1}
D_{\mathcal{C}}\circ g(\beta_j) &= (d_A)\beta_j-\omega\wedge \dma\beta_j+\theta\left[(d_A)\dma\beta_j-\Phi\beta_j\right]\nonumber\\
&= \dpa\beta_j+\theta\left[\dpa\dma\beta_j-\Phi\beta_j\right].
\end{align}
On the other hand, $m'_1(\beta_j)=\dpa\beta_j$ and 
$$g\circ m'_1(\beta_j)=\dpa\beta_j-\theta \dma\dpa\beta_j = \dpa\beta_j + \th\left[(d_A)\dma\beta_j-\Phi\beta_j\right]$$
using \eqref{dpmrelb}. Therefore, $D_{\mathcal{C}}\circ g(\beta_j) =g\circ m'_1(\beta_j).
$

\noindent \textbf{Case 2}. $j=n$.
For $\beta_n\in\mathcal{F}^n$, we have from \eqref{gc1} that
\begin{align*}
D_{\mathcal{C}}\circ g(\beta_n)=\theta\left[\dpa\dpm\beta_n-\Phi\beta_n\right]\
\end{align*}
having noted that $\dpa\beta_n=0$.  On the other hand, $m'_1(\beta_n)=-\dpa\dma\beta_n+\Phi\beta_n$ and we find $g\circ m'_1(\beta_n)=\theta\left[\dpa\dpm\beta_n-\Phi\beta_n\right] =D_{\mathcal{C}}\circ g(\beta_n)$.

\noindent \textbf{Case 3}. $j>n$.
For  $\beta_k\in P^k_-(M,E)=\mathcal{F}^j$ where $k=2n+1-j$, we have $g(\beta_k)=-\theta\,\omega^{n-k}\wedge\beta_k$ and
\begin{align*}
D_{\mathcal{C}}\circ g(\beta_k)=-\omega^{n-k+1}\wedge\beta_k+\theta\,\omega^{n-k}\wedge(d_A)\beta_k=\theta\,\omega^{n-k+1}\wedge\dma\beta_k.
\end{align*}
where we have used the primitive property that $\om^{n-s+1}\w \beta_s =0$ for $\beta_s\in P^s$.
On the other hand, $m'_1(\beta_k)=-\dma\beta_k$ and  we find $g\circ m'_1(\beta_k) =  \theta\, \omega^{n-k+1} \w \dma\beta_k = D_{\mathcal{C}}\circ g(\beta_k)$.
\end{proof}

With $f$ and $g$ being chain maps, we now proceed to relate the cohomologies associated with $\CF^*$ and $\CC^*$.  Define the operator $G:\mathcal{C}^j\to\mathcal{C}^{j-1}$  
\begin{align*}
G(\eta_j+\theta\,\eta_{j-1})=\eta_{j-1}+\theta L^{-1}\eta_j,
\end{align*}
for any $\eta_j\in\Omega^j(M,E)$ and $\eta_{j-1}\in\Omega^{j-1}(M,E)$. 
We have the following the result.
\begin{lem}\label{gfhom}
The maps $f, g$ and $G$ are related as follows: 
\begin{align*}
fg=\id_\mathcal{F}\,, \qquad 
\id_\mathcal{C}-gf-\Phi=D_{\mathcal{C}}G+GD_{\mathcal{C}}\,.
\end{align*}
\end{lem}
\begin{proof}
That $fg=\id_\mathcal{F}$ follows immediately from the definitions of the maps in Definition \ref{fgdef}. For the second relation, we consider first the left-hand side.  We have from the definitions
\begin{align*}
gf(\alpha_j) = 
\begin{cases}
\beta_j-\theta\, \dma\beta_j & 0\leq j\leq n \\
\th\, \om^{n-k}\w \left(\be_k + \dpa \be_{k_1}\right) & n+1\leq j\leq 2n+1,\, k=2n+1-j
\end{cases}
\end{align*}
and therefore,
\begin{align*}
(\id_\mathcal{C}-gf-\Phi)(\alpha_j) = 
\begin{cases}
\alpha_j-\Phi\alpha_j-\beta_j+\theta\, \dma\beta_j & 0\leq j\leq n \\
\alpha_j-\Phi\alpha_j-\theta\,\omega^{n-k}\wedge\left(\beta_k+\dpa\beta_{k-1}\right) & n+1\leq j\leq 2n+1,\, k=2n+1-j
\end{cases}
\end{align*}
%\noindent \textbf{Case 1}. $j\leq n$.  
%For $\alpha\in\mathcal{C}^j$, it follows from the definitions that $gf(\alpha_j)=\beta_j-\theta\, \dma\beta_j$, and therefore,
%\begin{align*}
%(\id_\mathcal{C}-gf-\Phi)(\alpha_j)=\alpha_j-\Phi\alpha_j-\beta_j+\theta\, \dma\beta_j.
%end{align*}
As for the right-hand side, we have for all $j$ 
\begin{align*}
G\alpha_j&=\eta_{j-1} + \th\, \LI\eta_j\\
D_{\mathcal{C}}G\alpha_j&=(d_A)\eta_{j-1} + \om \w L^{-1}\eta_j - \th\left[(d_A)(\LI\eta_j)+ \Phi \eta_{j-1}\right]\\
%\end{align*}
%and
%\begin{align*}
D_{\mathcal{C}}\alpha_j &= 
(d_A)\eta_j + \om \w \eta_{j-1} -\th\left[(d_A)\eta_{j-1} + \Phi \eta_j\right]\\
GD_{\mathcal{C}}\alpha_j &= -(d_A)\eta_{j-1} -\Phi\eta_j + \th\left[\LI(d_A) \eta_j + \LI(\om\w\eta_{j-1})\right]
\end{align*}
which implies
\begin{align}\label{rhseq}
(D_{\mathcal{C}}G+GD_{\mathcal{C}})\alpha_j &= \om \w \LI\eta_j -\Phi \alpha_j + \th\left[\LI(d_A)\eta_j-(d_A)\LI\eta_j+\LI(\om\w \eta_{j-1})\right]
\end{align}
When $j\leq n$, this implies using the form decomposition in \eqref{conek1} and the properties of $(d_A)$ in \eqref{dAdef} and \eqref{dAcom} that
\begin{align*}
(D_{\mathcal{C}}G+GD_{\mathcal{C}})\alpha_j &=\left(\eta_j-\be_{j}\right) -\Phi \alpha_j + \th\left[ \dma \be_j + \eta_{j-1}\right] \\
&=(\alpha_j-\be_j) -\Phi\alpha_j + \th\, \dma\be_j = (\id_\mathcal{C}-gf-\Phi)(\alpha_j).
\end{align*}
And when $j>n$, with $k=2n+1-j$ so that $\alpha_j=\eta_{2n+1-k} + \theta\,\eta_{2n-k}$, \eqref{rhseq} together with the form decomposition in \eqref{conek2} implies
\begin{align*}
(D_{\mathcal{C}}G+GD_{\mathcal{C}})\alpha_j &=\eta_j - \Phi\alpha_j + \th\left[-\om^{n-k}\w\dpa\be_{k-1} + (\eta_{j-1}-\om^{n-k}\w \be_k) \right]\\
&= \alpha_j -\Phi\alpha_j - \th \left[\om^{n-k} \w (\dpa\be_{k-1}+\be_k)\right] =(\id_\mathcal{C}-gf-\Phi)(\alpha_j).
\end{align*}
\end{proof}

\begin{lem}\label{phi exact}
For arbitrary $\alpha\in \mathcal{C}$, if $\alpha$ is $D_{\mathcal{C}}$-closed, then $\Phi\alpha$ is $D_{\mathcal{C}}$-exact.
\end{lem}
\begin{proof}
Suppose $\alpha\in\mathcal{C}^j$, we can set $\alpha=\eta_j+\theta\eta_{j-1}$ where $\eta_j,\eta_{j-1}\in\Omega^*(M,E)$. Since $\alpha$ is $D_{\mathcal{C}}$-closed,
$$
0=D_{\mathcal{C}}\alpha=(d_A\eta_j+\omega\eta_{j-1})-\theta(\Phi\eta_j+d_A\eta_{j-1}).
$$
So $d_A\eta_{j-1}=-\Phi\eta_j$. Then we have
$$
D_{\mathcal{C}}(-\eta_{j-1})=-d_A\eta_{j-1}+\theta\Phi\eta_{j-1}=\Phi(\eta_j+\theta\eta_{j-1}).
$$
\end{proof}

Now let $f^*$ and $g^*$ be the induced maps between $H_\CC^*(M,E)$ and $PH^*(M,E)$. Then, from Lemma \ref{gfhom} and \ref{phi exact}, we have 
\begin{align*}
f^*g^*&=\id \qquad \text{ on } \,PH^*(M,E),\\
g^*f^*&=\id \qquad \text{ on }  \quad H_\CC^*(M,E),
\end{align*}
which imply the following theorem.

\begin{thm}
There is an isomorphism between the twisted cohomologies:
$$
PH^*(M,E)\cong H_\CC^*(M,E)\,.
$$
\end{thm}
From the local cohomologies of $PH^*(U, E)$ in Theorem \ref{localP}, we have the corollary:
\begin{cor}
In a local coordinate chart $U$, there exists some $\lambda\in\Omega^1(U)$ such that $d\lambda=\omega$. The local cohomology
$$
H_\CC^j(U,E)=
\begin{cases}
\ker\Phi, & j=0 \\
(\lambda-\theta)\ \mathrm{coker}\ \Phi, & j=1 \\
0, & j\geq 2
\end{cases}
$$
\end{cor}
And finally, with Theorem \ref{vanP} which states the triviality of $PH^*(M,E)$ when $\Phi$ is invertible, we also have the following:
\begin{cor}
When $\Phi$ is invertible,  $H_\CC^*(M,E)=0$.
\end{cor}

%%%%%%%%%%%%%%%%%%%%%%%%%%%%%%%%%%%%

%%%%%%%%%%%%%%%%%%%%%%%%%%%%%%%%%%%%

%\begin{thebibliography}{99}

\vskip 1cm
\noindent
%Li-Sheng Tseng\\
{Department of Mathematics, University of California, Irvine, CA 92697, USA}\\
{\it Email address:}~{\tt lstseng@math.uci.edu}
\vskip .5 cm
\noindent
%Jiawei Zhou\\
{Yau Mathematical Sciences Center, Tsinghua University, Beijing, 100084, China}\\
{\it Email address:}~{\tt jiaweiz1990@mail.tsinghua.edu.cn}

\end{document}